\documentclass[12pt]{amsart}

\usepackage{amssymb}

\newtheorem{theorem}{Theorem}
\newtheorem{lemma}[theorem]{Lemma}

\theoremstyle{definition}
\newtheorem{definition}[theorem]{Definition}

\theoremstyle{remark}

\numberwithin{equation}{section}

\newcommand{\eps}   {\varepsilon}
\newcommand{\tensor}{\otimes}

\newcommand{\CC}{\mathbb{C}}

\newcommand{\F}{\mathcal{F}}

\DeclareMathOperator{\ad}{ad}
\DeclareMathOperator{\At}{At}
\DeclareMathOperator{\End}{End}
\DeclareMathOperator{\GL}{GL}
\DeclareMathOperator{\rank}{rank}
\DeclareMathOperator{\rel}{rel}
\DeclareMathOperator{\tr}{tr}
\DeclareMathOperator{\vol}{vol}

\renewcommand{\leq}{\leqslant}
\renewcommand{\geq}{\geqslant}

\begin{document}

\baselineskip=15pt

\title[Approximate Hermitian--Einstein structures]{Existence of 
approximate Hermitian--Einstein structures on semistable 
principal bundles}

\author[I.~Biswas]{Indranil Biswas}

\address{School of Mathematics, Tata Institute of Fundamental
Research, Homi Bhabha Road, Bombay 400005, India}

\email{indranil@math.tifr.res.in}

\author[A.~Jacob]{Adam Jacob}

\address{Department of Mathematics, Columbia University, New 
York, NY 10027, USA}

\email{ajacob@math.columbia.edu}

\author[M.~Stemmler]{Matthias Stemmler}

\address{School of Mathematics, Tata Institute of Fundamental
Research, Homi Bhabha Road, Bombay 400005, India}

\email{stemmler@math.tifr.res.in}



\date{}

\begin{abstract}
Let $E_G$ be a principal $G$--bundle over a compact 
connected K\"ahler 
manifold, where $G$ is a connected reductive linear algebraic 
group defined over $\CC$. We show that $E_G$ is semistable if 
and only if it admits approximate Hermitian--Einstein structures.
\end{abstract}

\maketitle

\section{Introduction}

A holomorphic vector bundle $E$ over a compact connected K\"ahler 
manifold $(X \, , \omega)$ is said to admit approximate 
Hermitian--Einstein metrics if for every $\eps \,>\, 0$, 
there is a Hermitian metric $h$ on $E$ such that
\[
  \sup_X \big|\sqrt{-1} \Lambda_\omega F(h) - \lambda \cdot
{\rm Id}_E\big|_h < \, \eps \, .
\]
In \cite[Theorem 2]{Ja11} it was shown that a holomorphic vector 
bundle $E$ over a compact K\"ahler manifold $(X \, , \omega)$ is 
semistable if and only if it admits approximate 
Hermitian--Einstein metrics. This generalizes a result of 
Kobayashi \cite[p.\ 234, Theorem 10.13]{Ko87} for complex 
projective manifolds. It is an analogue, for semistable bundles, 
of the classical Hitchin--Kobayashi correspondence, which is 
given by the famous theorem of Donaldson, Uhlenbeck and Yau. 
This theorem relates polystable bundles to (exact) solutions of 
the Hermitian--Einstein equation, and was first proven for curves 
by Narasimhan and Seshadri \cite{NS65}, then for algebraic 
surfaces by Donaldson \cite{Do85}, and finally in arbitrary 
dimension by Uhlenbeck and Yau \cite{UY86}.

Our aim here is to generalize the above result of \cite{Ja11} to 
principal $G$--bundles over $X$, where $G$ is a connected 
reductive linear algebraic group defined over $\CC$. Fix a 
maximal compact subgroup $K\, \subset\, G$.
A holomorphic principal $G$--bundle $E_G$ over $X$
is said to admit approximate Hermitian--Einstein structures if 
for every $\eps \,>\, 0$, there exists a $C^\infty$ reduction 
of structure group $E_K \,\subset\, E_G$ to $K$ and an
element $\lambda$ in the center of $\text{Lie}(G)$ such that
\[
\sup_X \big|\Lambda_\omega F(\nabla^{E_K}) - 
\lambda\big|_{h_{E_K}} < \, \eps \, ,
\]
where $F(\nabla^{E_K})$ is the curvature form of the Chern 
connection of $E_K$, and $h_{E_K}$ is the Hermitian metric on 
$\ad(E_G)$ induced by $E_K$ (this Hermitian metric $h_{E_K}$
is described in \eqref{induced}).

We prove the following:

\begin{theorem}\label{main}
A holomorphic principal $G$--bundle $E_G$ over $X$ is semistable 
if and only if it admits approximate Hermitian--Einstein 
structures.
\end{theorem}

\section{Preliminaries}

Let $(X \, , \omega)$ be a compact connected K\"ahler manifold of complex dimension $n$, and let $E$ be a holomorphic vector bundle over $X$.

Recall that the {\em degree\/} of a torsion--free coherent 
analytic sheaf $\F$ on $X$ is defined to be
\[
  \deg(\F) \, := \, \int_X c_1(\F) \wedge \omega^{n-1} \, ;
\]
if $\rank(\F) > 0$, the {\em slope\/} of $\F$ is
\[
  \mu(\F) \, := \, \frac{\deg(\F)}{\rank(\F)} \, .
\]

\begin{definition}
A holomorphic vector bundle $E$ is called {\em semistable\/} if 
$\mu(\F) \, \leq \, \mu(E)$ for every nonzero coherent 
analytic subsheaf $\F$ of $E$.
\end{definition}

\begin{definition}
A holomorphic vector bundle $E$ is said to admit {\em approximate 
Hermitian--Einstein metrics\/} if for every $\eps\,>\, 0$, there 
exists a Hermitian metric $h$ on $E$ such that
\[
\sup_X \big|\sqrt{-1} \Lambda_\omega F(h) - \lambda \cdot 
{\rm Id}_E\big|_h 
< \, \eps \, .
\]
Here $\Lambda_\omega$ is the adjoint of the wedge 
product with $\omega$, $F(h)$ is the curvature form of the Chern 
connection for $h$, and $\lambda$ is given by
$$
  \lambda \, = \, \frac{2 \pi \cdot \mu(E)}{(n-1)! \cdot \vol(X)} \, ,
$$
where $\vol(X)$ denotes the volume of $X$ with respect to 
the K\"ahler form $\omega$.
\end{definition}

In \cite{Ja11}, the following was proved:
\begin{theorem}[{\cite[Theorem 2]{Ja11}}] \label{jacob}
A holomorphic vector bundle $E$ over $X$ is semistable if and 
only if it admits approximate Hermitian--Einstein metrics.
\end{theorem}

Now let $G$ be a connected reductive linear algebraic group 
defined over $\CC$, and let $E_G$ be a holomorphic principal 
$G$--bundle over $X$.

\begin{definition}
$E_G$ is called {\em semistable\/} if for every triple $(P \, , 
U \, , \sigma)$, where
\begin{itemize}
\item $P \,\subset \, G$ is a maximal proper parabolic subgroup,
\item $U \subset X$ is a dense open subset such that the complement $X \setminus U$ is a complex analytic subset of $X$ of codimension at least $2$, and
\item $\sigma: U \longrightarrow E_G/P$ is a holomorphic 
reduction of structure group, over $U$, of $E_G$ to the subgroup 
$P$, satisfying the condition that the pullback $\sigma^\ast 
T_{\rel}$, which is a holomorphic vector bundle over $U$, 
extends to $X$ as a coherent analytic sheaf (here $T_{\rel}$ is 
the relative tangent bundle over $E_G/P$ for the natural 
projection $E_G/P \longrightarrow X$),
\end{itemize}
the inequality
\[
  \deg(\sigma^\ast T_{\rel}) \, \geq \, 0
\]
holds. The degree of $\sigma^\ast T_{\rel}$ is
\[
\deg(\sigma^\ast T_{\rel}) \, := \, \int_X c_1(\iota_\ast \sigma^\ast T_{\rel}) \wedge \omega^{n-1} \, ,
\]
where $\iota: U \longrightarrow X$ is the inclusion map.
\end{definition}

We will now define approximate Hermitian--Einstein structures on $E_G$. Let
\begin{equation} \label{atiyah}
0\,\longrightarrow\, \ad(E_G)\,\longrightarrow\, \At(E_G)\,
\stackrel{q}{\longrightarrow}\, TX\,\longrightarrow\, 0
\end{equation}
be the Atiyah exact sequence for $E_G$ (see \cite{At57} for the 
construction of \eqref{atiyah}). Recall that a {\em 
complex connection\/} on $E_G$ is a $C^\infty$ splitting
$D\, :\, TX\,\longrightarrow\, \At(E_G)$ of this 
exact sequence, meaning $q\circ D\,=\, \text{Id}_{TX}$.
Note  that \eqref{atiyah} is a short exact sequence of sheaves of 
Lie algebras. The {\em curvature form\/} of a 
connection $D$,
$$
  F(D) \, \in \, H^0(X \, , \Lambda^{1,1} T^\ast X \tensor 
\ad(E_G)) \, ,
$$
measures the obstruction of the homomorphism $D$ to be Lie 
algebra structure preserving; see \cite{At57} for the details.

Fix a maximal compact subgroup
$$
K \, \subset\, G\, .
$$
A {\em Hermitian structure\/} on $E_G$ is a smooth reduction of 
structure group $E_K$ of $E_G$ to $K$. Given a Hermitian 
structure $E_K$ on $E_G$, there is a unique complex connection on 
$E_G$ which is induced by a connection on $E_K$. This connection on 
$E_G$ is called the {\em Chern connection\/} of the Hermitian 
structure $E_K$, and it will be denoted by $\nabla^{E_K}$.
The connection on $E_K$ inducing the Chern connection on
$E_G$ will also be called the Chern connection. Let
\begin{equation} \label{curvature}
  F(\nabla^{E_K}) \, \in \, H^0(X \, , \Lambda^{1,1} T^\ast X 
\tensor \ad(E_G))
\end{equation}
be the curvature of $\nabla^{E_K}$. Note that
$F(\nabla^{E_K})$ lies in the image of
$\Lambda^{2} (T^{\mathbb R} X)^\ast\tensor \ad(E_K)$, where
$T^{\mathbb R} X$ is the real tangent bundle.

Let $\mathfrak{g}$ be the Lie algebra of $G$.
Consider the adjoint representation
\begin{equation} \label{adjoint}
  \rho: \, G \,\longrightarrow\, \GL(\mathfrak{g}) \, .
\end{equation}
Fix a maximal compact subgroup $\widetilde K\,\subset\, 
\GL(\mathfrak{g})$ containing $\rho(K)$. Let 
$$
E_{\GL(\mathfrak{g})}\,=\, E\times^G 
\GL(\mathfrak{g})\,\longrightarrow\, X
$$
be the principal $\GL(\mathfrak{g})$--bundle obtained by 
extending the structure group of $E_G$ to $\GL(\mathfrak{g})$ 
using the homomorphism $\rho$ in \eqref{adjoint}. We note that
the vector bundle associated to the principal 
$\GL(\mathfrak{g})$--bundle $E_{\GL(\mathfrak{g})}$ for the 
standard action of $\GL(\mathfrak{g})$ on $\mathfrak{g}$ is 
identified with the adjoint bundle $\ad(E_G)$.

Given a Hermitian structure $E_K \subset E_G$, we obtain a reduction of structure group
\begin{equation} \label{induced}
E_K(\widetilde K) \, = \, E_K \times^K \widetilde K \, \subset 
\, E_{\GL(\mathfrak{g})}
\end{equation}
of $E_{\GL(\mathfrak{g})}$ to $\widetilde K$. This reduction 
corresponds to a Hermitian metric on the adjoint vector bundle 
$\ad(E_G)$.

Let $\mathfrak{z}$ be the center of the Lie algebra $\mathfrak{g}$. Since the adjoint action of $G$ on $\mathfrak{z}$ is trivial, an element $\lambda \in \mathfrak{z}$ defines a smooth section of $\ad(E_G)$, which will also be denoted by $\lambda$.

\begin{definition} \label{approximate}
A holomorphic principal $G$--bundle $E_G$ over $X$ is said to 
admit {\em approximate Hermitian--Einstein 
structures\/} if for every $\eps \,>\, 0$, there exists a 
Hermitian structure $E_K \,\subset\, E_G$ and an element 
$\lambda\, \in\, \mathfrak{z}$, such 
that
\[
\sup_X \big|\Lambda_\omega F(\nabla^{E_K}) - 
\lambda\big|_{h_{E_K}} < \, \eps \, ,
\]
where $F(\nabla^{E_K})$ is the curvature form of the Chern 
 connection of $E_K$ (see \eqref{curvature}), and $h_{E_K}$ is 
the Hermitian metric on $\ad(E_G)$ induced by $E_K$ (see 
\eqref{induced}).
\end{definition}

\section{Proof of Theorem \ref{main}}

We will first show that it is enough to prove the theorem under 
the assumption that $G$ is semisimple.

Let $Z_0(G)$ be the connected component of the center of $G$
which contains the identity element. The normal subgroup $[G \, , 
G] \,\subset\, G$ is semisimple because $G$ is reductive. We have 
a natural surjective homomorphism
\[
G \,\longrightarrow \, (G/Z_0(G)) \times (G/[G \, , G])
\]
whose kernel is a finite group contained in the center of $G$.
In particular, the induced homomorphism of Lie algebras is an
isomorphism.

Let $\rho\,:\, A \,\longrightarrow\, B$ be a homomorphism of Lie 
groups 
such that the induced homomorphism of Lie algebras
$$
d\rho\, :\, \text{Lie}(A)\, \longrightarrow\, \text{Lie}(B)
$$
is an isomorphism, and $\text{kernel}(\rho)$ is contained in the
center of $A$. Let $E_A$ be a principal $A$--bundle, and let $E_B 
\,:=\, E_A \times^\rho B$ be the principal $B$--bundle obtained 
by extending the structure group of $E_A$ to $B$ using $\rho$. 
The isomorphism of Lie algebras $d\rho$ produces 
an isomorphism
$$
\widetilde{\rho}\,:\, \text{ad}(E_A) 
\,\longrightarrow\,\text{ad}(E_B)
$$
between the adjoint bundles.
There is a natural bijective correspondence between the 
connections on $E_A$ and the connections on $E_B$. To construct
this bijection, first note that $\rho$ induces a map
$$
\widehat{\rho}\, :\, E_A\, \longrightarrow\, E_B
$$
that sends $z\, \in\, E_A$ to the element of $E_B$ given by
$(z\, , e)$ (recall that $E_B$ is a quotient of $E_A\times B$).
This map $\widehat{\rho}$ intertwines the actions of $A$, with
$A$ acting on $E_B$ through~$\rho$. Since $\text{kernel}(\rho)$ 
is a finite group contained in the center of $A$, any
$A$--invariant vector field on $E_A\vert_U$, where $U$ is some
open subset of the base manifold, produces a $B$--invariant 
vector field on $E_B\vert_U$. This way we get an isomorphism of
$\text{At}(E_A)$ with $\text{At}(E_B)$. This identification
of $\text{At}(E_A)$ with $\text{At}(E_B)$ produces a bijection
between the connections on $E_A$ and the connections on 
$E_B$. The curvature of a connection on $E_B$ is given by the 
curvature of the corresponding connection on $E_A$ using the 
isomorphism $\widetilde{\rho}$.

Therefore, to prove the theorem, it 
suffices to prove it for $G/Z_0(G)$ and $G/[G \, , G]$ 
separately. Since $G/[G \, , G]$ is a product of copies of 
$\CC^\ast$, in this case the theorem follows immediately from 
Theorem \ref{jacob}. Since $G/Z_0(G)$ is semisimple, it is enough 
to prove the theorem under the assumption that $G$ is semisimple.

Henceforth, we will assume that $G$ is semisimple. This implies 
that the center $\mathfrak z$ of its Lie algebra $\mathfrak{g}$ 
is trivial, and thus the constant $\lambda$ in Definition 
\ref{approximate} is zero. The Killing form on $\mathfrak{g}$,
being $G$--invariant, produces a holomorphic bilinear form
on the fibers of $\ad(E_G)$. Since the Killing form is
nondegenerate (as $G$ is semisimple), this bilinear form
on the fibers of $\ad(E_G)$ is nondegenerate. Hence we get
a trivialization
\begin{equation}\label{tr}
\det (\ad(E_G))\, :=\, \bigwedge\nolimits^{\rm top} \ad(E_G)\, 
\stackrel{\sim}{\longrightarrow}\, {\mathcal O}_X\, .
\end{equation}
Therefore, $\deg(\ad(E_G))\,=\, \deg(\det (\ad(E_G)))
\,=\, 0$, or equivalently
\begin{equation}\label{m}
\mu(\ad(E_G))\,=\, 0\, .
\end{equation}

First assume that the principal bundle $E_G$ admits approximate 
Hermitian--Einstein structures. Given $\eps \,>\, 0$, we thus 
obtain a Hermitian structure $E_K \,\subset\, E_G$ satisfying the 
condition that
\begin{equation} \label{epsilon}
  \sup_X \big|\Lambda_\omega F(\nabla^{E_K})\big|_{h_{E_K}} < \, \eps \, .
\end{equation}
The Hermitian structure $E_K$ on $E_G$ induces a Hermitian 
metric $h_{E_K}$ on $\ad(E_G)$ (see \eqref{induced}). The Chern 
connection on $\ad(E_G)$ for $h_{E_K}$ coincides with the 
connection $\nabla^{\ad}$ on $\ad(E_G)$ induced by 
$\nabla^{E_K}$. The curvature forms of $\nabla^{E_K}$ and $\nabla^{\ad}$ are related by
\begin{equation} \label{curvatures}
  F(\nabla^{\ad}) \, = \, \ad(F(\nabla^{E_K})) \, ,
\end{equation}
where
\begin{equation}\label{adh}
\ad\,:\, \ad(E_G) \,\longrightarrow\, \End(\ad(E_G))\,=\,
\ad(E_G)\otimes \ad(E_G)^*
\end{equation}
is the homomorphism of vector bundles induced by the homomorphism
of Lie algebras ${\mathfrak g}\, \longrightarrow\,{\mathfrak 
g}\otimes {\mathfrak 
g}^*$ given by the adjoint action of $G$ on $\mathfrak{g}$.

Since the Hermitian metric $h_{E_K}$ on $\ad(E_G)$ is induced by 
a Hermitian structure on $E_G$, there is a real number $c_0 > 0$
such that $\frac{1}{c_0} \cdot \ad$ is an isometry, where $\ad$
is the homomorphism in \eqref{adh}. Now from
\eqref{curvatures} it follows that
\begin{align*}
  \big|\sqrt{-1} \Lambda_\omega F(\nabla^{\ad})\big|_{h_{E_K}}^2
  &= \, \tr(\Lambda_\omega F(\nabla^{\ad}) \circ (\Lambda_\omega F(\nabla^{\ad}))^\ast) \\
  &= \, \tr(\ad(\Lambda_\omega F(\nabla^{E_K})) \circ (\ad(\Lambda_\omega F(\nabla^{E_K})))^\ast) \\
  &= \, c^2_0\cdot h_{E_K}(\Lambda_\omega F(\nabla^{E_K}) \, , 
\Lambda_\omega F(\nabla^{E_K})) \\
  &= \, c^2_0\cdot \big|\Lambda_\omega 
F(\nabla^{E_K})\big|_{h_{E_K}}^2 \, ,
\end{align*}
where ``$\ast$'' denotes the adjoint with respect to $h_{E_K}$. From this and 
\eqref{epsilon} we conclude that
\[
  \sup_X \big|\sqrt{-1} \Lambda_\omega 
F(\nabla^{\ad})\big|_{h_{E_K}} < \, c^2_0\cdot \eps \, .
\]
Therefore, $\ad(E_G)$ admits approximate Hermitian--Einstein 
metrics, and is semistable by Theorem ~\ref{jacob}.
A holomorphic principal $G$--bundle $F_G$ over $X$ 
is semistable if and only if its adjoint vector bundle $\ad(F_G)$ 
is semistable \cite[Proposition 2.10]{AB01}. Therefore,
a principal $G$--bundle admitting approximate
Hermitian--Einstein structures is semistable.

For the converse direction, assume that $E_G$ is semistable. As 
we have stated, this is equivalent to the vector bundle 
$\ad(E_G)$ being semistable. Let ${\mathcal H}(\ad(E_G))$ be the 
space of all $C^\infty$ Hermitian metrics $h$ on $\ad(E_G)$ 
satisfying the following condition: the isomorphism in
\eqref{tr} takes the Hermitian metric on $\det (\ad(E_G))$ 
induced by $h$ to the constant Hermitian metric on ${\mathcal 
O}_X$ given by the absolute value. For any initial $h\, \in\, 
{\mathcal H}(\ad(E_G))$, we can evolve the metric by the 
following parabolic equation, which we call the Donaldson heat 
flow:
$$
h^{-1} \partial_t h \,=\, - (\sqrt{-1} \Lambda_\omega F(h)
-\lambda\cdot \text{Id}_{\ad(E_G)})\, .
$$
Here $F(h)$ is the curvature of the Chern connection on
$\ad(E_G)$ for $h$, and
\[
  \lambda \, = \, \frac{2\pi \cdot \mu(\ad(E_G))}{(n-1)! \cdot \vol(X)} \, .
\]
Since $\mu(\ad(E_G))\,=\,0$ (see \eqref{m}), for us the Donaldson heat 
flow is given by the simpler expression
\begin{equation}\label{df}
h^{-1} \partial_t h \,=\, - \sqrt{-1} \Lambda_\omega F(h)\, .
\end{equation}

As shown in \eqref{induced}, a Hermitian structure on $E_G$ 
produces a Hermitian metric on $\ad(E_G)$. Such a Hermitian 
metric on $\ad(E_G)$ satisfies the condition that the isomorphism 
in \eqref{tr} takes the induced Hermitian metric on $\det 
(\ad(E_G))$ to the constant Hermitian metric on ${\mathcal O}_X$ 
given by the absolute value. In other words, it lies in
${\mathcal H}(\ad(E_G))$. Let
$$
{\mathcal H}(E_G)\, \subset\, {\mathcal H}(\ad(E_G))
$$
be the subspace corresponding to the Hermitian structures on 
$E_G$.

\begin{lemma}\label{lem1}
The Donaldson heat flow on ${\mathcal H}(\ad(E_G))$ preserves
${\mathcal H}(E_G)$.
\end{lemma}

\begin{proof}
Let $E_K\, \subset\, E_G$ be a $C^\infty$ reduction of structure 
group to $K$. The element of ${\mathcal H}(\ad(E_G))$ given by
the Hermitian structure $E_K$ will be denoted by $h$.
The Chern connection $\nabla(h)$ on $\ad(E_G)$ for $h$ is given 
by the Chern connection $\nabla^{E_K}$ on $E_K$. In 
particular, the curvature of $\nabla(h)$ coincides with that of
$\nabla^{E_K}$. Therefore, the curvature $F(h)$ in \eqref{df}, 
which is a priori a real two--form 
with values in $\End(\ad(E_G))$, is actually a real two--form 
with values in ${\rm ad}(E_K)$ (the adjoint bundle of $E_K$).
Consequently, $\Lambda_\omega F(h)$ is a $C^\infty$ section of
${\rm ad}(E_K)$ (the operator $\Lambda_\omega$ takes real 
two--forms to real valued functions). This implies 
that the Donaldson heat flow on 
${\mathcal H}(\ad(E_G))$ preserves ${\mathcal H}(E_G)$.
\end{proof}

Fix a Hermitian metric $h_0 \in {\mathcal H}(\ad(E_G))$, and consider the 
Donaldson heat flow with initial metric $h_0$. The space ${\mathcal H}(\ad(E_G))$ was defined so that $h_0$ satisfies the normalization
\[
  c_1(\ad(E_G) \, , h_0) \, = \, \frac{\sqrt{-1}}{2 \pi} \, \tr(F(h_0)) \, = \, 0 \, .
\]
This guarantees that
\[
  \det(h_0^{-1} h) \, = \, 1
\]
along the flow. From this and the semistability of $\ad(E_G)$,
it follows that approximate Hermitian--Einstein 
metrics on $\ad(E_G)$ are realized along the flow for 
sufficiently large time (see the proof of Theorem~\ref{jacob} 
in \cite{Ja11} for details). Consequently, taking $h_0$ to be an 
element of ${\mathcal 
H}(E_G)$, from Lemma \ref{lem1} we conclude that $E_G$ admits 
approximate Hermitian--Einstein structures. This completes the
proof of Theorem \ref{main}.

\end{document}